\documentclass[11pt]{amsart}
\usepackage{amssymb, mathrsfs, comment, xcolor}
\usepackage{graphicx}
\usepackage[showonlyrefs]{mathtools}
\usepackage[colorlinks=true, allcolors=black]{hyperref}
\usepackage[a4paper, hmargin=2.5cm, vmargin={3.7cm, 3.7cm}]{geometry}
\usepackage{titlesec}
\numberwithin{equation}{section}
\newcommand{\periodafter}[2]{#1{#2.}}
\newcommand*{\justifyheading}{\raggedright}
\titleformat {\section}
{\bfseries\justifyheading\MakeUppercase}
{\thesection}{2em}{\periodafter{}}
\titleformat {\subsection}
{\bfseries\justifyheading}{\thesubsection}{1em}{\periodafter{}}

\usepackage{esint}
\usepackage{graphicx}
\usepackage{bm}
\usepackage{comment}
\usepackage{epsfig}

\usepackage{mathrsfs}
\usepackage{enumitem}
\usepackage{color}
\usepackage{amsfonts}
\usepackage{verbatim}
\usepackage{amsthm}
\usepackage{dsfont}
\usepackage{url}
\usepackage{esint}
\usepackage{tocloft}
\usepackage[mathcal]{euscript}
\theoremstyle{definition}
\newtheorem{definition}{Definition}[section]

\theoremstyle{remark}
\newtheorem{remark}[definition]{Remark}
\newtheorem{problem}[definition]{Problem}

\theoremstyle{plain}
\newtheorem{theorem}[definition]{Theorem}

\newtheorem{lettertheorem}{Theorem}

\newcommand{\cj}{\overline}

\newcommand{\N}{\mathbb{N}}
\newcommand{\C}{\mathbb{C}}

\numberwithin{equation}{section}

\begin{document}

\title{Chebyshev polynomials on equipotential curves}

\author{Erwin Mi\~{n}a-D\'{\i}az}
\address{Department of Mathematics,
The University of Mississippi,
Hume Hall 305, P. O. Box 1848,
MS 38677-1848, USA}
\email{minadiaz@olemiss.edu}

\author{Olof Rubin}
\address{Department of Mathematics,
KU Leuven,
Celestijnenlaan 200B,
3001 Leuven, Belgium}
\email{olof.rubin@kuleuven.be}

\date{\today} 



\maketitle
\begin{abstract}For an analytic function $\phi(z)$ with a Laurent expansion at $\infty$ of the form 
	\begin{equation*}
		\phi(z)=z+c_{0}+\frac{c_{1}}{z}+\frac{c_{2}}{z^{2}}+\cdots,
	\end{equation*}
the Faber polynomial $F_n$ of degree $n$ associated to $\phi$ is the polynomial part of the Laurent series at $\infty$ of $\phi(z)^n$. We prove that the $n$th Chebyshev polynomial $T_{n,L_r}$ for the equipotential curve $L_r=\{z\in \C:|\phi(z)|=r \}$  converges to $F_n$ as $r\to\infty$. The proof makes use of the fact that zero is the strongly unique best approximation to the monomial $z^n$ on the unit circle by polynomials of degree less than $n$. 

\end{abstract}
\section{Introduction}
If $K\subset \C$ is a compact set with infinitely many points, then for every integer $n\geq 0$, there exists a unique monic polynomial $T_{n,K}$ of degree $n$ that minimizes the supremum norm over $K$:
\[
\max_{z\in K}|T_{n,K}(z)|=\inf\left\{\max_{z\in K}|a_0+a_1z+\cdots+a_{n-1}z^{n-1}+z^n|:\ a_0,\ldots,a_{n-1}\in \C\right\}.
\]
The polynomial $T_{n,K}$ is known as the $n$th-degree Chebyshev polynomial associated with the compact set $K$ and is named after Pafnuty L. Chebyshev, who originally studied the minimal polynomials on the interval $[-1,1]$; see \cite{chebyshev54}. The extension of this theory to general compact subsets of the complex plane is due to Faber \cite{faber19}. Since then, the literature on Chebyshev polynomials in the complex plane has expanded significantly, and several sources offer detailed treatments of their fundamental properties. We refer the reader to \cite{achieser56, christiansen-simon-zinchenko-I, novello-schiefermayr-zinchenko21, rivlin90, widom69} for a more extensive background treatment of Chebyshev polynomials.

If $K$ is a circle about the origin, then the Chebyshev polynomials are given by $T_{n,K}(z) = z^n$, see \cite{faber19}. If $K=[-1,1]$, then
\begin{align}\label{Chebyshev-poly}
T_{n,[-1,1]}(z)=\left(\frac{z+\sqrt{z^2-1}}{2}\right)^n+\left(\frac{z-\sqrt{z^2-1}}{2}\right)^n
\end{align}
is the classical $n$th-degree Chebyshev polynomial of the first kind \cite{chebyshev54}. 

Let $\Omega$ denote the unbounded component of $\cj{\C}\setminus K$. The Maximum Modulus Principle implies that $T_{n,K}=T_{n,\partial \Omega}$. Thus, the Chebyshev polynomials for a compact set coincide with those for its outer boundary, and if desired one may always assume that $\Omega=\cj{\C}\setminus{K}$. If $K$ has positive logarithmic capacity $\operatorname{cap}(K)>0$, then $\Omega$ has a Green function $g_{\Omega}(z,\infty)$ with pole at $\infty$, which  is uniquely characterized by the properties
\begin{enumerate}
	\item[(a)] $g_{\Omega}(z,\infty)$ is harmonic in $\Omega\setminus\{\infty\}$, and bounded outside each neighborhood of $\infty$;
		\item[(b)]  $\lim_{z\to\infty} g_{\Omega}(z,\infty) - \log |z|=-\log \operatorname{cap}(K)$;
	\item[(c)]  $g_{\Omega}(z,\infty)\to 0$ as $z\to \zeta\in \partial\Omega$, for nearly every $\zeta\in \partial \Omega$. 
\end{enumerate}
We refer the reader to \cite{garnett-marschall05, ransford95} for details on Green functions, logarithmic capacity and other related notions in potential theory.

Our focus will be directed toward Chebyshev polynomials corresponding to the level curves of the Green function  
\[
L_r\coloneqq\{z\in \Omega: g_{\Omega}(z,\infty)=\log r\},\qquad r>1,
\]
which are typically referred to as \textit{equipotential curves}. Because of the harmonicity of the Green function one can construct an analytic function $\phi$ on $\Omega$ such that (see e.g.  \cite[Sec. 3]{widom69})
\begin{align}
g_{\Omega}(z,\infty)=\log|\phi(z)|,\qquad z\in \Omega.
\label{green-function-map}
\end{align}
Unless $\Omega$ is simply-connected, the function $\phi$ is multiple-valued. However, $|\phi(z)|$ remains single-valued across $\Omega$, which allows us to write
\begin{align}\label{levelcurves-secondrepresentation}
L_r = \{ z \in \Omega : |\phi(z)| = r \},\qquad r>1.
\end{align}
Moreover, $\phi$ is always univalent in a neighborhood of $\infty$ with a Laurent series expansion of the form
\begin{equation}
	\phi(z) = c z + c_0 + \frac{c_{-1}}{z} + \frac{c_{-2}}{z^2} + \cdots, \qquad c = \operatorname{cap}(K)^{-1},
	\label{eq:conformal_representation-alt}
\end{equation}
so that for $r$ sufficiently large the level curve $L_r$ is an analytic Jordan curve. If in addition $\Omega$ is simply-connected, then $\phi$ is the canonical conformal map of $\Omega$ onto the exterior of the unit circle. For a pictorial illustration of how equipotential curves $L_r$ may look in relation to the compact set $K$, see Figure \ref{fig:bernoulli_equipotential}.

\begin{figure}[h!]
	\centering
	\includegraphics[width = 0.7\textwidth]{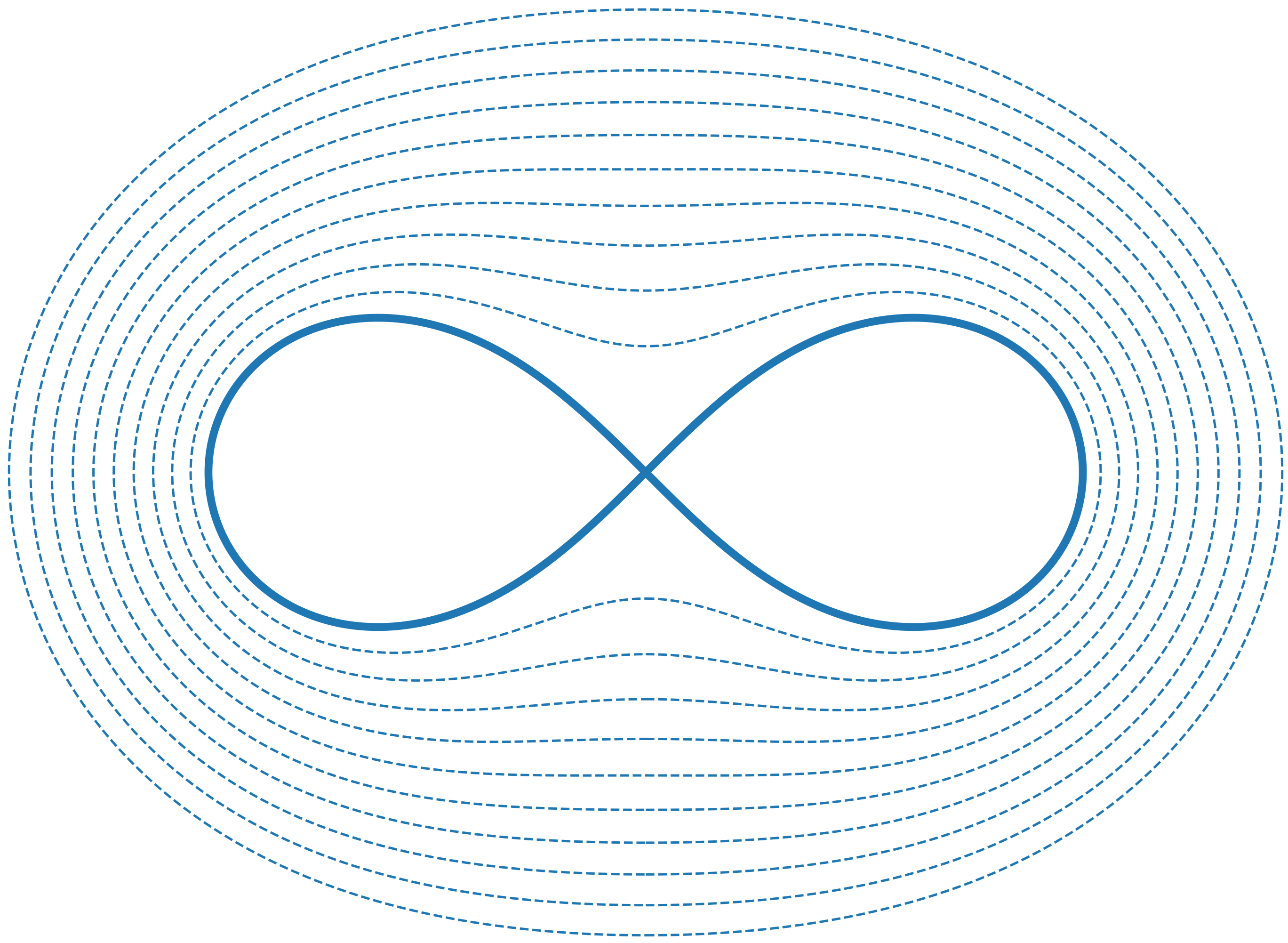}	
	\caption{The Bernoulli lemniscate $K = \{z: |z^2-1| = 1\}$ as a solid curve along with samples of associated equipotential curves $L_r = \{z:|z^2-1| = r^2\}$ for $r>1$.}
	\label{fig:bernoulli_equipotential}
\end{figure}

In certain instances, it has been observed that the Chebyshev polynomial $T_{n,L_r}$ for the level curve $L_r$ coincides with $T_{n,K}$ for all $r>1$. A simple example arises when $K$ is the unit circle, then $L_r$ is the circle about the origin of radius $r$. In this case, $T_{n,K}(z) = T_{n,L_r}(z)=z^n$ for all $r>1$. This is also true for general lemniscates.

\begin{lettertheorem}[Faber 1919 \cite{faber19}]
	Let $P$ be a monic polynomial of degree $m$ and let $R>0$. If we set
	\[K\coloneqq \{z\in \mathbb{C}: |P(z)|=R\},\]
	then 
	\[T_{nm,K}(z) = T_{nm,L_r}(z) = P(z)^{n}\]
	for all $n\in\N$ and $r>1$.
	\label{thm:faber_lemniscate}
\end{lettertheorem}
The proof is omitted from Faber's original article since it is a simple consequence of the Maximum Modulus Principle. Faber showed that a similar situation arises for ellipses.

\begin{lettertheorem}[Faber 1919 \cite{faber19}]
	For $K=[-1,1]$, the level curves $L_r$ are ellipses with foci at $-1$ and $1$, and  \[T_{n,L_r}(z)=T_{n,[-1,1]}(z)\]
	holds for all $n\in \mathbb{N}$ and $r>1$.
	\label{thm:faber_interval}
\end{lettertheorem}

It turns out that this is just a special case of a phenomenon that holds more generally for so-called \emph{period-$n$ sets}. A compact set $K$ is called a \emph{period-$n$ set} if there exists a real polynomial $P$ of degree $n$ and real points $x_0<x_1<\cdots <x_n$ satisfying $P(x_k) = (-1)^{n-k}$ such that
\[
P^{-1}([-1,1]) = K.
\]
We refer the reader to \cite{geronimo-vanassche88, peherstorfer93, totik01} for more on the basic properties of such sets. 

\begin{lettertheorem}[Christiansen, Simon and Zinchenko 2020 \cite{christiansen-simon-zinchenko-IV}]
	If $K$ is a period-$n$ set for some $n$, then \[T_{nm,K}(z) = T_{nm,L_r}(z)\] for every $m\in\N$ and $r>1$.	
	\label{thm:csz_period_n_set}
\end{lettertheorem}
This result was first established for period-$n$ compact sets consisting of exactly two intervals by Fischer \cite{fischer92}. Yet another regularity result holds true for Julia sets of quadratic polynomials.

\begin{lettertheorem}[Stawiska \cite{stawiska96}]
	Let $K = J_{P_\lambda}$ be the Julia set of the polynomial $P_{\lambda}(z) = (z-\lambda)^2$ for $0\leq \lambda \leq 2$ then 
	\[T_{2^{n},K}(z) = T_{2^n,L_r}(z)\]
	holds for all $n\in\N$ and $r>1$.
	\label{thm:stawiska_julia}
\end{lettertheorem}

This was further refined in \cite[Theorem 2]{xiao-qiu09} to hold for Julia sets generated by $P(z)=z^d-a$ with $d\geq 2$ and $a\in \C$.

In this paper we demonstrate that the regularity observed in Theorems \ref{thm:faber_lemniscate}, \ref{thm:faber_interval}, \ref{thm:csz_period_n_set}, and \ref{thm:stawiska_julia} always manifests in the limit as $r\to\infty$.

\section{Main result}

  We will formulate our result  solely on  the basis of \eqref{levelcurves-secondrepresentation} and \eqref{eq:conformal_representation-alt}, without reference to a compact set $K$. Our departure point is therefore a function $\phi$ that is analytic in a neighborhood of $\infty$ with a Laurent expansion of the form  
\begin{equation}
	\phi(z) = c z + c_0 + \frac{c_{-1}}{z} + \frac{c_{-2}}{z^2} + \cdots,\qquad c>0.
	\label{eq:conformal_representation}
\end{equation}
For each integer $n \geq 0$, the polynomial part $F_{n}(z)$ of the Laurent expansion at $\infty$ of $\phi(z)^{n}$ is  called \textit{the Faber polynomial} of degree $n$ associated to $\phi$, see \cite{minadiaz06}. We denote by $\hat F_n$ the $n$th monic Faber polynomial, that is, $\hat F_n(z)=c^{-n}F_n(z)$.

The level curves 
	\[
	L_{r}=\{z \in \C:|\phi(z)|=r\} 
	\]
	are all well-defined for $r$ sufficiently large. Let $T_{n,L_r}$ be the Chebyshev polynomial of degree $n$ for $L_{r}$.
The following regularity result was initially suggested by numerical experiments conducted in \cite{hubner-rubin24}.

\begin{theorem}
For all $n \geq 0$,
\begin{equation}
\lim _{r \rightarrow \infty} T_{n,L_r}(z)=\hat F_{n}(z), \qquad z \in \mathbb{C}.
\label{eq:main_eq}
\end{equation}
\label{main:thm}
\end{theorem}

The limit holds with respect to any norm on the finite-dimensional space of polynomials of degree at most $n$, as all such norms are equivalent. However, we will establish the stronger result that
\begin{equation}
	\max_{z\in L_r}|T_{n,L_r}(z)-\hat{F}_n(z)| = O(r^{-1})
\end{equation}
as $r\rightarrow \infty$. A visual representation of this convergence is presented in Figure \ref{fig:trajectory} where we illustrate the zeros of the Chebyshev polynomial of degree $21$ corresponding to varying equipotential curves of the Bernoulli lemniscate $K=\{z\in \mathbb{C}: |z^2-1| = 1\}$. Note that the even Chebyshev polynomials can be determined from Theorem \ref{thm:faber_lemniscate} and are given by $T_{2n,K}(z) = (z^2-1)^n$. In particular, they have all their zeros situated at $\pm 1$. It is a result of Ullman \cite{ullman60} that the zeros of the odd Faber polynomials corresponding to the Bernoulli lemniscate lie inside the lemniscate except for the simple zero at the origin.
\begin{figure}[h!]
	\centering
	\includegraphics[width = 0.7\textwidth]{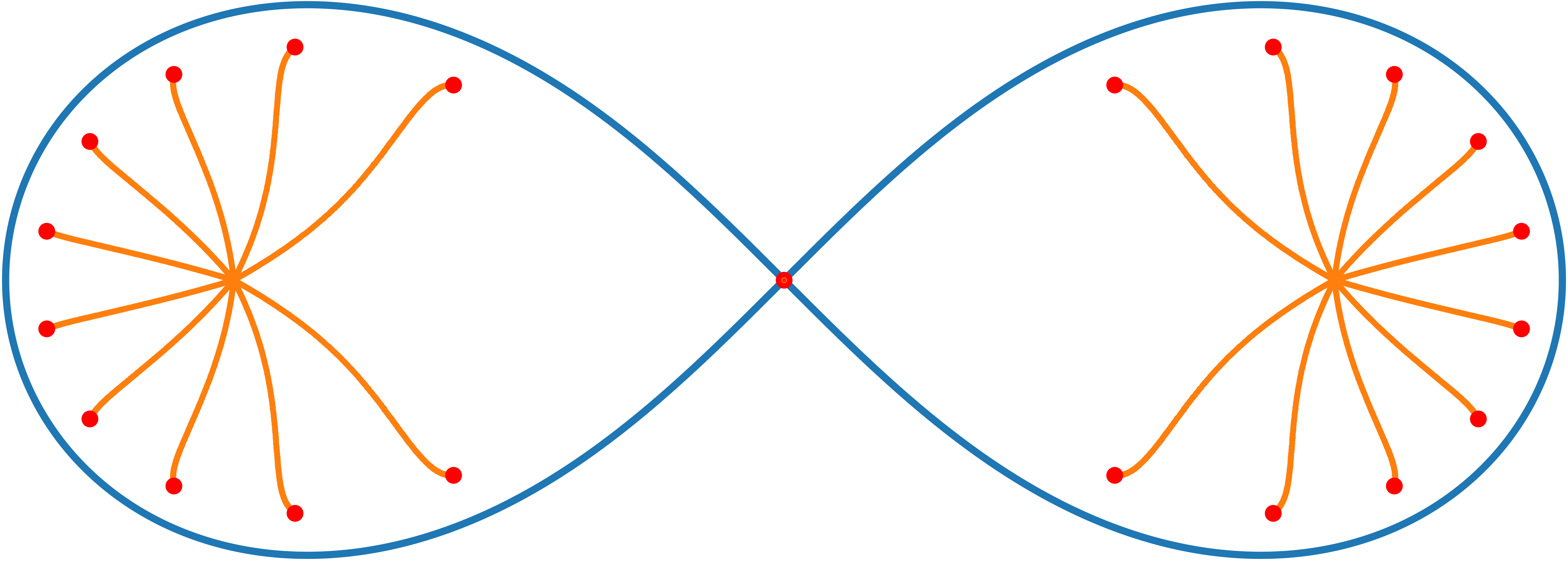}
	\caption{The orange lines illustrate the trajectories of the zeros of the Chebyshev polynomials $ T_{21,L_r} $ corresponding to the level curves $L_r = \{ z \in \mathbb{C} : |z^2 - 1| = r^2 \}$ as $r$ increases from $0$. The red dots mark the zeros of the corresponding Faber polynomial. The Bernoulli lemniscate $L_1\coloneqq \{z\in \mathbb{C}: |z^2-1| = 1\}$ is marked with a solid blue line. The computations were carried out using the algorithm presented in \cite{tang87, tang88}.}
	\label{fig:trajectory}
\end{figure}

We have mentioned examples of compact sets $K$ with the property that for certain values of $n$ (sometimes all of them), $T_{n,L_r}=T_{n,K}$ for all $r>1$. Theorem \ref{main:thm} shows that when this equality happens then necessarily the $n$th Chebyshev polynomial and the $n$th monic Faber polynomial are one and the same for the compact set $K$. Thus, it is natural to pose the following question.
\begin{problem}
Suppose the compact set $K$ is such that for some $n\in \N$, $T_{n,K}=\hat F_n$, where $\hat F_n$ is the $n$th monic Faber polynomial associated to the map $\phi$ that satisfies \eqref{green-function-map}. Is it true then that  $T_{n,L_r}=\hat F_n$ for all $r>1$?	
\end{problem}

\begin{remark}
	While we focus on the case where the degree $n$ of the Chebyshev polynomial is fixed and the equipotential curves vary, there is another classical setting in which the degree varies while the level curve remains fixed. In particular, a connection between the asymptotics of Chebyshev polynomials as $n \to \infty$ and Faber polynomials was established in \cite{faber19}, see also \cite[Theorem~8.3]{widom69}. It is shown that if $r$ is fixed and sufficiently large so that $L_r$ is an analytic Jordan curve, then
	\[
	\lim_{n \to \infty} \left( \frac{c}{r} \right)^{n} \max_{z \in L_r} \left| T_{n, L_r}(z) - \hat{F}_n(z) \right| = 0.
	\]
	Here, the quantity $r/c$ is the logarithmic capacity of $L_r$, and provides a natural normalization.
\end{remark}

The proof of Theorem \ref{main:thm} will make use of the fact that zero is the \emph{strongly unique best approximation} to the monomial $z^n$ on the unit circle by polynomials of degree less than $n$. Let $\mathrm{P}$ denote a finite dimensional subspace of the linear space $C(K)$ of complex-valued continuous functions on a compact set $K\subset \C$, endowed with the supremum norm $\|\cdot\|_K$. Given $f\in C(K)$, an element $p_0\in \mathrm{P}$ is said to be a strongly unique best approximation to $f$ out of $\mathrm{P}$ provided that there exists a constant $\gamma>0$ such that 
\begin{equation}
	\|f-p_0\|_K+\gamma\|p-p_0\|_K\leq\|f-p\|_K,\qquad p\in \mathrm{P}.
	\label{eq:su}
\end{equation}
The \textit{strong uniqueness constant} for a function $f$ with respect to $K$ and $\mathrm{P}$ is the largest possible value of $\gamma$ such that \eqref{eq:su} holds. This constant is given by
\[
\inf_{p \in \mathrm{P} \setminus \{p_0\}} \frac{\|f - p\|_K - \|f - p_0\|_K}{\|p - p_0\|_K};
\]
see \cite{kroo-pinkus10} for a survey on strong unicity.

In \cite{rivlin84}, Rivlin studied the case where $K=\{z:|z|\leq 1\}$ is the closed unit disk, $f$ is the monomial $z^n$, and $\mathrm{P}=\textrm{P}_{n-1}$ is the space of polynomials of degree at most $n-1$ for some $n\in \N$. Observe that since $z^n$ is the Chebyshev polynomial of degree $n$ for the unit circle, the constant zero function is the unique best approximation to $z^n$ from $\textrm{P}_{n-1}$. Rivlin found that this best approximant is actually strongly unique, and that the strong uniqueness constant for $z^n$ is $1/n$. By the Maximum Modulus Principle,  this can be stated as follows:
 \begin{theorem}[Rivlin 1984 \cite{rivlin84}] For every $n\in \N$, 
 \begin{align}\label{Rivlin-ineq-1}
\max_{|z|= 1}|p(z)|\leq n\left(\max_{|z|= 1}|p(z)+z^n|-1\right)	,\qquad p\in \mathrm{P}_{n-1}.
\end{align}
 	\label{thm:Rivlin}
 \end{theorem}
This will be a key result in establishing Theorem \ref{main:thm}.

\section{Proofs}
We assume for convenience that  $c=1$ in \eqref{eq:conformal_representation}, that is, 
\begin{equation}
	\phi(z)=z+c_{0}+\frac{c_{1}}{z}+\frac{c_{2}}{z^{2}}+\cdots.
	\label{eq:normalized_conformal_representation}
\end{equation}
The Faber polynomial $F_{n}$ for $\phi$ is then already monic, and by its own definition, we have
\begin{equation}
	F_{n}(z)-\phi(z)^{n}=O(1 / z), \quad z \rightarrow \infty.
	\label{eq:faber_error}
\end{equation}
The function $\phi$ is conformal in a neighborhood of $\infty$. We use the letter  $\psi$ to denote the inverse of $\phi$. Since our interest lies in the behavior of Chebyshev polynomials on $L_r$ for large $r$ we may also assume for convenience that $\psi$ admits a conformal extension to $\Delta_1:=\{z\in \C:|z|>1\}$. Then $\phi$ extends conformally to the domain $\Omega_1=\psi(\Delta_1)$, and the equipotential curves 
\[
L_r:=\{z \in \Omega_1:|\phi(z)|=r\},\qquad r>1 
\]
are all analytic Jordan curves. 

By the defining property of $T_{n, L_r}$, we know that for every $r>1$,
\begin{equation}
\max _{z \in L_{r}}\left|T_{n, L_r}(z)\right| \leq\max _{z \in L_{r}}\left|F_{n}(z)\right|.
\end{equation}
Pick a point $z_{r}$ of $L_{r}$ such that
$$
\left|F_{n}\left(z_{r}\right)\right|= \max _{z \in L_{r}}\left|F_{n}(z)\right|.
$$
From \eqref{eq:normalized_conformal_representation} and \eqref{eq:faber_error} we have as $r \rightarrow \infty$
$$
\frac{F_{n}\left(z_{r}\right)}{\phi\left(z_{r}\right)^{n}}=1+O\left(r^{-n-1}\right).
$$
Therefore, when $r \rightarrow \infty$
\begin{equation}
\max _{z \in L_{r}}\left|r^{-n}T_{n, L_r}(z)\right| \leq 1+O\left(r^{-n-1}\right).
\label{eq:cheb_norm_bound}
\end{equation}
Theorem \ref{main:thm} follows from \eqref{eq:cheb_norm_bound} as a consequence of the following more general statement.
\begin{theorem}Let $n\in \N$ be fixed. Suppose that for every $r>1$, $Q^r_n$ is a monic polynomial of degree $n$ and that there exists a constant $M>0$ such that 
\begin{align} \label{eq:cheb_norm_bound-repeat}
\max _{z \in L_{r}}\left|r^{-n}Q_{n}^{r}(z)\right| \leq 1+Mr^{-n-1}, \qquad r>1.
\end{align}
Then there exists a constant $\tilde M>0$ such that if $r\geq 2$
\begin{equation}
	\max_{z\in L_r}|Q_n^r(z)-F_n(z)|\leq \tilde Mr^{-1}.
	\label{eq:strong_convergence}
\end{equation}
\label{thm:general}
\end{theorem}

\begin{proof}
The monic polynomial $Q_n^r$ may be expanded in the Faber basis in the following form 
\begin{align}\label{expansion-Faber}
Q_{n}^{r}(z)=F_{n}(z)+\sum_{k=0}^{n-1} \alpha_{k}^{r} F_{k}(z).
\end{align}
We will begin by proving the existence of a constant $A$ such that  
\begin{align}\label{alphas-coeff-ineq}
|\alpha^r_{k}|\leq Ar^{-k-1}, \qquad 0\leq k\leq n-1
\end{align}
when $r$ is bounded away from $1$, say $r\geq 2$ for convenience.

The function 
\[
G_n^r(w):=r^{-n}Q_n^r(\psi(rw)), \qquad |w|>r^{-1}
\] is analytic on $|w|>r^{-1}$ and has a Laurent expansion of the form
\begin{align*}
G_n^r(w)=\sum_{k=-\infty}^{n-1}b^r_kw^k+w^n,\qquad |w|>r^{-1}.
\end{align*}
Making $z=\psi(rw)$ in \eqref{expansion-Faber} and using \eqref{eq:faber_error} we see that as $w\to\infty$
\[
G_n^r(w)=w^n+\sum_{k=0}^{n-1} \alpha_{k}^{r} r^{k-n}w^k+O(1/w). 
\]
Thus,
\begin{align}\label{equality-alpha-b}
\alpha_k^r=b_k^rr^{n-k},\qquad 0\leq k\leq n-1.
\end{align}

Note that because $Q_n^r$ is a polynomial, the Maximum Modulus Principle implies that the inequality \eqref{eq:cheb_norm_bound-repeat} also holds for $z$ in the interior domain of $L_r$. In particular, we have
\begin{align}\label{ineq-max-princ}
	1< |\phi(z)| \leq r \implies \left|r^{-n}Q_{n}^{r}(z)\right| \leq 1+Mr^{-n-1}. 
 \end{align}

Making the replacement $z=\psi(rw)$ in \eqref{ineq-max-princ} yields
\begin{equation}\label{eq:cheb_norm_bound-repeat1}
\left|G_n^r(w)\right| \leq 1+Mr^{-n-1},\qquad r^{-1}<|w|\leq 1.
\end{equation}
It follows that for $\rho$ in the range $r^{-1}<\rho<1$ we have by Parseval's formula
\begin{align*}
	\sum_{k=-\infty}^{n-1}|b^r_k|^2\rho^{2k}+\rho^{2n}=\frac{1}{2\pi}\int_{0}^{2\pi} |G_n^r(\rho e^{i\theta})|^2d\theta\leq (1+Mr^{-n-1})^2.
\end{align*}
Letting $\rho\to r^{-1}$ we find 
\begin{align*}
	\sum_{k=-\infty}^{n-1}|b^r_k|^2r^{-2k}+r^{-2n}\leq (1+Mr^{-n-1})^2.
\end{align*}
In particular, 
\[
|b^r_k|\leq (1+Mr^{-n-1})r^k, \qquad -\infty<k\leq n-1.
\]
Thus, on the unit circle we have the bound 
\begin{align}\label{bound-circle}
\left|\sum_{k=-\infty}^{-n-1}b_k^rw^k\right|\leq\frac{r^{-n-1}(1+Mr^{-n-1})r}{r-1},\qquad |w|=1.
\end{align}
Making use of  \eqref{eq:cheb_norm_bound-repeat1} and \eqref{bound-circle} we find that with, say $M_1=2+3M$, we have
  \begin{align}\begin{split}
\left|\sum_{j=0}^{2n-1}b_{j-n}^rw^{j}+w^{2n}\right|={} &|w^n|\left|\sum_{k=-n}^{n-1}b_k^rw^k+w^n\right|\\
\leq{} & |G_n^r(w)|+\left|\sum_{k=-\infty}^{-n-1}b_k^rw^k\right|\\
	  \leq {} & 1+M_1r^{-n-1},\qquad |w|=1, \quad r\geq 2. 
	  \end{split}
\end{align}

We now employ Rivlin's Inequality \eqref{Rivlin-ineq-1} to obtain 
\begin{align*}
	\sup_{|w|=1}\left|\sum_{j=0}^{2n-1}b_{j-n}^rw^{j}\right|\leq{} & 2n\left(\sup_{|w|=1}\left|\sum_{j=0}^{2n-1}b_{j-n}^rw^{j}+w^{2n}\right|-1\right)\\
	\leq {} & 2nM_1r^{-n-1}.
\end{align*}	
Thus, for $j=0,1,\ldots,2n-1$, another application of Parseval's formula gives us that 
\begin{align*}
\sum_{k=-n}^{n-1}|b_k^r|^2=\frac{1}{2\pi}\int_{0}^{2\pi}\left|\sum_{j=0}^{2n-1}b_{j-n}^re^{ij\theta}\right|^2d\theta\leq (2nM_1 r^{-n-1})^2,
\end{align*}
which together with \eqref{equality-alpha-b} yields 
\begin{align}
	|\alpha_k^r|=|b_k^rr^{n-k}|\leq 2nM_1 r^{-k-1},\qquad 0\leq k\leq n-1,\quad r\geq 2.
	\label{eq:alpha_estimate}
\end{align}
Having \eqref{alphas-coeff-ineq} proven, we now turn to establishing \eqref{eq:strong_convergence}. From \eqref{eq:faber_error} we see that there exists a constant $M_2$ such that
\[
\max_{z\in L_r}|F_k(z)| \leq r^{k}+ M_2 r^{-1},\qquad 0\leq k\leq n-1,\quad r\geq 2.
\]
 Using \eqref{eq:alpha_estimate} we find that for all $r\geq 2$ and $z\in L_r$ 
\[|Q_n^r(z)-F_n(z)| = \left|\sum_{k=0}^{n-1}\alpha_k^rF_k(z)\right|\leq 2nM_1\sum_{k=0}^{n-1}(r^{-1}+M_2r^{-k-2})\leq 2nM_1\left(n+\frac{M_2}{r-1}\right)r^{-1},
\]
which establishes \eqref{eq:strong_convergence}.
\end{proof}

Theorem \ref{thm:general} is stronger than Theorem \ref{main:thm} in the sense that not only does the limit in \eqref{eq:main_eq} hold locally uniformly, but it actually holds uniformly on the growing level curves $L_r$. Since
\[\max_{z\in L_r}|F_n(z)-\phi(z)^n| = O(r^{-1})\]
as $r\rightarrow \infty$, we can just as well replace $F_n$ by $\phi^n$ in \eqref{eq:strong_convergence}.

\end{document}